\documentclass[11pt]{amsart}
\usepackage{amssymb,url}
\usepackage{hyperref}
\usepackage{times}
\renewcommand{\epsilon}{\varepsilon}
\newcommand{\pa}{$\mbox{p.} \, $}

\newcommand{\gln}{\mathrm{GL}(n,\mathbb{F})}

\newcommand{\glnq}{\mathrm{GL}(n,q)}
\newcommand{\glz}{\mathrm{GL}(n,\mathbb{Z})}
\newcommand{\glq}{\mathrm{GL}(n,\mathbb{Q})}

\newcommand{\Z}{\mathbb{Z}}
\newtheorem{theorem}{Theorem}[section]
\newtheorem{lemma}[theorem]{Lemma}
\newtheorem{proposition}[theorem]{Proposition}
\newtheorem{corollary}[theorem]{Corollary}
\theoremstyle{definition}

\theoremstyle{example}
\newtheorem{example}[theorem]{Example}
\theoremstyle{remark}
\newtheorem{remark}[theorem]{Remark}

\begin{document}

\title[Computing with Nilpotent Matrix Groups 
over Infinite Domains]
{Algorithms for computing with nilpotent matrix 
groups over infinite domains}

\author{A. S. Detinko \and D. L. Flannery}

\begin{abstract}
We develop methods for computing with matrix groups defined over a 
range of infinite domains, and apply those methods to the design of 
algorithms for nilpotent groups. In particular, we provide a practical 
algorithm to test nilpotency of matrix groups over an infinite field. 
We also provide algorithms that answer a number of structural questions 
for a given nilpotent matrix group.
 The algorithms have been implemented in {\sf GAP} and {\sc Magma}.
\end{abstract}

\maketitle

\section{Introduction}

In this paper we develop a technique for computing with matrix
groups defined over infinite domains, based on changing the ground
domain via congruence homomorphism. This technique has proved to
be very efficient in linear group theory (see, e.g.,
\cite[Chapter 3]{Dixon}). It is especially useful for 
finitely generated linear groups (see \cite[Chapter
4]{Wehrfritz}), and affords a general approach to many
computational problems for infinite matrix groups. 

Let $\mathbb{F}$ be a field and let $G\leq \gln$ be given by a
finite generating set. We obtain algorithms to carry out the
following tasks:
\begin{itemize}
\item[(i)] testing nilpotency of $G$;
\end{itemize}
and, if $G$ is nilpotent,
\begin{itemize}
\item[(ii)] constructing a  polycyclic presentation of $G$;
\item[(iii)] testing whether $G$ is completely reducible, and
finding a completely reducible series of $G$-modules; 
\item[(iv)]
deciding finiteness of $G$, and calculating $|G|$ if $G$ is
finite; 
\item[(v)] finding the $p$-primary decomposition of $G$,
and finding all Sylow $p$-subgroups if $G$ is finite.
\end{itemize}
These algorithms address standard problems in computational group
theory (i, ii, iv) and computing with matrix groups (iii),
 and facilitate structural analysis of nilpotent linear
groups (v).

Our guiding objective has been to design algorithms that cover the
broadest possible range of infinite domains. However, for
convenience or reasons of practicality we sometimes
restrict $\mathbb{F}$. For example, a preliminary reduction
in nilpotency testing assumes that $\mathbb{F}$ is perfect, and
constructing polycyclic presentations requires $\mathbb{F}$ to be
finite or an algebraic number field. 
The implementation and practicality of our
algorithms relies on the machinery (such as polynomial
factorization) that is available for computing with the various
fields.

A finitely generated subgroup 
of $\gln$ lies in
$\mathrm{GL}(n,R)$ for some finitely 
generated subring $R$ of $\mathbb F$.
In turn, a completely reducible
solvable subgroup of $\mathrm{GL}(n,R)$ is finitely generated 
(\cite[Theorem 6.4, p.~111]{Dixon}, \cite[4.10, 
p.~57]{Wehrfritz}). Thus, our algorithms supply 
information not only about 
nilpotent subgroups of $\gln$, but
also about nilpotent subgroups of $\mathrm{GL}(n,R)$ for any
finitely generated integral domain $R$. Furthermore, these
algorithms serve as a platform for computing in abstract
finitely generated nilpotent groups, because such a group is
isomorphic to a subgroup of $\glz$ for some $n$. 
A method to construct a
representation over $\Z$ of a finitely presented polycyclic group
may be found in \cite{MR1716420}.

Nilpotency is an important group-theoretic property, and testing
nilpotency is consequently one of the basic functions of any
computational group theory system. We provide the first
uniform and effective solution to the problem of computing with
infinite nilpotent matrix groups. Our algorithms for nilpotency
testing (over finite fields, and $\mathbb{Q}$) have been
implemented as part of the {\sf GAP} package `Nilmat'~\cite{Nilmat} 
(this is joint work with Bettina Eick). Previous
algorithms for nilpotency testing in {\sf GAP}~\cite{Gap} and 
{\sc Magma}~\cite{Magma}
sometimes fail to decide nilpotency even for small finite matrix
groups, and fail for almost all infinite matrix groups. In the
paper's final subsection we give sample experimental results and
details of the `Nilmat' package.


This paper is a slight revision of \cite{FirstNilInf}, which was a first 
step in adapting the method of finite approximation 
as a computational tool. We have made substantial progress since
\cite{FirstNilInf}, particularly with regard to solvable-by-finite groups, 
including finite and solvable as special cases: see \cite{BigSurvey}.              
%
Also note that a
\href{http://magma.maths.usyd.edu.au/magma/handbook/matrix_groups_over_infinite_fields}{{\sc Magma} implementation}
of algorithms in this paper by Eamonn O'Brien
handles more kinds of domains than \cite{Nilmat}.

\section{Related results}

Computing in matrix groups over an infinite domain is a relatively
new area of computational group theory. Most of the algorithms in
this area are concerned with classes of solvable-by-finite groups
(see \cite{BettinaAssmann,BettinaAssmann2,Beals,Ostheimer}).
Solvable-by-finite groups constitute the more optimistic class of
the Tits alternative. The other class consists of groups
that contain a non-abelian free subgroup. For those groups, some
basic computational problems, such as membership testing and
construction of presentations, are undecidable (see
\cite{Beals,MR780623,Bettinasurvey}).

Changing the ground domain is a common technique in linear group
theory. It has been used by
several authors for computing with matrix groups; see, e.g.,
\cite{Luks}. In \cite{Beals}, a generalization of the technique as
in \cite{Luks} leads to a Monte-Carlo solvability testing
algorithm for potentially infinite subgroups $G$ of $\glq$. The
algorithm accepts as input a finite set $S$ of generators of $G$,
and tests solvability of $\psi_p(G)\leq \mathrm{GL}(n,p)$, where
$\psi_p$ is reduction modulo a prime $p$ not dividing the
denominators of the entries of the elements in $S \cup S^{-1}$. 
There are only finitely many primes
$p$ such that $\psi_p(G)$ is solvable while $G$ is not; so a
non-solvable group will be identified as solvable by the algorithm
of \cite{Beals} with small probability.

The ideas of \cite{Beals} might be applied to nilpotency testing.
However, 
 the upper bound on nilpotency
class for nilpotent subgroups of $\glnq$ can be much larger than
the bound for nilpotent subgroups of $\glq$ (see
\cite{Bialo,MR1878590}). Hence the solvability testing arguments
of \cite{Beals} may not be efficient when applied to nilpotency
testing, if one simply replaces bounds depending on derived length
by bounds depending on nilpotency class.

Let $G$ be a finitely generated matrix group  over $\mathbb{Q}$.
We can test solvability of $G$ if we can 
test solvability both of the kernel $G_p$ and of the
image $\psi_p(G)$ of a reduction mod $p$ homomorphism $\psi_p$.
Theoretical background for doing this is laid out in
\cite{MR780623}, where $G_p$ is described for solvable-by-finite
$G\leq \glq$. Using those results, a deterministic algorithm for
solvability testing was proposed in \cite{Ostheimer}.  There were
two main obstacles to a full implementation of the algorithm in
\cite{Ostheimer}: solvability testing of matrix groups over a
finite field, and efficient construction of $G_p$. Practical
solutions of these problems were obtained in
\cite{BettinaAssmann,assmann}. Specifically, \cite{BettinaAssmann}
contains a method to construct a polycyclic presentation of
$\psi_p(G)$, and thereby to test solvability of $\psi_p(G)$. The
relators of this presentation may be used to calculate generators
of a subgroup of $G$ whose normal closure is $G_p$. Although the
algorithm has bottlenecks (see \cite[\mbox{p.}
1281]{BettinaAssmann}), it has been successfully implemented for
solvability testing over finite and algebraic number fields (see
the {\sf GAP} package `Polenta'~\cite{Polenta}).

The aims of \cite{BettinaAssmann} are to test whether a
finitely generated subgroup $G$ of $\glq$ is polycyclic, and, if
so, to construct a polycyclic presentation for $G$. Those 
problems are solved in a subsequent 
publication~\cite{BettinaAssmann2}.  This provides an avenue 
for testing nilpotency of $G$: if $G$ is not polycyclic then 
it is not nilpotent; otherwise, nilpotency of $G$ can be tested 
using a polycyclic presentation of $G$ (for which see
 \cite[Section 4]{Lo}).

In this paper we propose an essentially different approach to
nilpotency testing, valid over a broad range of infinite
domains.  In contrast to \cite{BettinaAssmann2}, our algorithms do
not require {\it a priori} testing of polycyclicity and
computation of polycyclic presentations, and are designed directly
for nilpotency testing.

We use established linear group theory, chiefly
structural results for nilpotent linear groups (\cite[Chapter
VII]{Suprunenko2}, \cite{MR2232095,proccor}).  Accordingly, a
feature of our algorithms is that they return structural
information about input nilpotent groups. A full solution of the
problem of testing nilpotency over finite fields appears in
\cite{Large} (as we will see, much of \cite{Large} remains valid
over any field). Nilpotency testing is transferred to 
groups over a finite field by means of a congruence homomorphism with
torsion-free kernel; see Section~\ref{changingviacongh}. Other
methods that transfer nilpotency testing to the case of finite
groups are given in Subsection~\ref{firstuseofabseries}.

\section{Changing the ground domain via congruence homomorphism}
\label{changingviacongh}

In this section we present some results from linear group theory
that comprise the theoretical foundation of our algorithms.

First we set up some notation. Let $\Delta$ be an integral domain.
For any ideal $\varrho\subset \Delta$, the natural surjection
$\psi_\varrho: \Delta\rightarrow \Delta/\varrho$ extends entrywise
to a matrix ring homomorphism $\mathrm{Mat}(n,\Delta)\rightarrow$
$\mathrm{Mat}(n,\Delta/\varrho)$, and then restricts to a group
homomorphism $\mathrm{GL}(n,\Delta)\rightarrow
\mathrm{GL}(n,\Delta/\varrho)$, which we also denote
$\psi_\varrho$. The map $\psi_\varrho$ on $\mathrm{GL}(n,\Delta)$
is called a Minkowski or congruence homomorphism (\cite[\mbox{p.}
65]{Suprunenko2}), and its kernel is called a (principal)
congruence subgroup of $\mathrm{GL}(n,\Delta)$. We denote the
congruence subgroup corresponding to $\varrho$ by
$\mathcal{G}(n,\Delta,\varrho)$, or $\mathcal{G}_\varrho$
for short. 
If $G\leq \mathrm{GL}(n,\Delta)$ then $G_\varrho:= G\cap
\mathcal{G}_\varrho$. For an integer $m$ we write $m\in \varrho$
to mean that $m\cdot 1_{\Delta} \in \varrho$.

We are interested in $\Delta$ and ideals $\varrho$
such that $\mathcal{G}(n,\Delta,\varrho)$ is torsion-free
if $\mathrm{char} \, \Delta=0$, or each torsion element of
$\mathcal{G}(n,\Delta,\varrho)$ is unipotent if $\mathrm{char} \,
\Delta >0$. Such domains in characteristic zero are discussed in
\cite[Chapter III, Section 11]{Suprunenko2}. A slight modification
of the proofs in \cite{Suprunenko2} takes care of the positive
characteristic case. To keep the account here reasonably
self-contained, we give full proofs.

\begin{lemma}
\label{preufd} Let $\Delta$ be a unique factorization domain,
$q\in \Delta$ be irreducible, and $\varrho$ be the principal ideal
$q \Delta$ of $\Delta$.
 Suppose that $\mathcal{G}(n,\Delta,\varrho)$ has non-trivial
torsion elements. Then
\begin{itemize}
\item[{\rm (i)}] there is a unique prime $p\in \mathbb{Z}$ such
that $p\in \varrho;$ 
\item[{\rm (ii)}] 
$pb=-\sum_{i=2}^p{p \choose i} q^{i-1}b^i$ for some 
$b\in \mathrm{Mat}(n,\Delta);$ and 
\item[{\rm (iii)}] every torsion element of
$\mathcal{G}(n,\Delta,\varrho)$ has $p$-power order.
\end{itemize}
\end{lemma}
\begin{proof}
(Cf.~\cite[proof of Theorem 3, \mbox{pp.} 68-69]{Suprunenko2}.)
Let $h\in \mathcal{G}_\varrho$ be of prime order $p$. We have
$h=1_n+ qb$ for some $b\in \mathrm{Mat}(n,\Delta)$. Then
$$1_n=h^p = 1_n+pqb+\cdots + {p \choose i} q^ib^i + \cdots +
q^pb^p
$$
where the binomial coefficients are read modulo $\mathrm{char}\,
\Delta$. Hence
\begin{equation}
\label{binom}
 pb=-\sum_{i=2}^p{p \choose i}
q^{i-1}b^i
\end{equation}
and it follows that either $q$ divides $p$, or $q$ divides every
entry of $b$.

Suppose that $q$ does not divide $p$. Then for some integer
$\alpha \geq 1$, $q^\alpha$ divides every entry of $b$, whereas
$q^{\alpha+1}$ does not. But \eqref{binom} implies that
$q^{2\alpha+1}$ divides $pb$, a contradiction. Thus $q$ divides
$p$.

If $\mathcal{G}_{\varrho}$ contains a non-trivial element of
$p'$-order then it contains an element of prime order $r\neq p$.
By the preceding, then, $q$ divides both $p$ and $r$ and hence
divides $1=px+ry$ for some $x,y\in \mathbb{Z}$.
Since $q$ is not a unit by definition,
every torsion element of $\mathcal{G}_{\varrho}$ must be
 a $p$-element.
\end{proof}

\begin{proposition}
\label{ufdcase} Let $\Delta$, $q$, and $\varrho$ be as in
Lemma{\em ~\ref{preufd}}.
\begin{itemize}
\item[{\rm (i)}] If $\mathrm{char}\, \Delta=t>0$ then every
 torsion element of
$\mathcal{G}(n, \Delta, \varrho)$ is a $t$-element. \item[{\rm
(ii)}] Suppose that $\mathrm{char}\, \Delta=0$, $q$ does not
divide $2$, and $q^2$ does not divide $p$ for any prime $p\in
\mathbb{Z}$. Then $\mathcal{G}(n , \Delta, \varrho)$ is
torsion-free.
\end{itemize}
\end{proposition}
\begin{proof}
(Cf.~\cite[\mbox{pp.} 68-69]{Suprunenko2}.)

(i) \hspace{2pt} 
 This follows from parts (i) and (iii) of Lemma~\ref{preufd}.

(ii) \hspace{1pt} If $\mathcal{G}_\varrho$ has non-trivial
torsion then $p=qr$  for some odd prime $p$ and $r\in \Delta$ not
divisible by $q$. By Lemma~\ref{preufd} (ii), for some $b, c\in
\mathrm{Mat}(n,\Delta)$ we have $qrb=q^2b^2c$. Hence $q^{\alpha}$
divides every entry of $b$ for some $\alpha \geq 1$ such that
$q^{\alpha+1}$ does not divide every entry of $b$. As $q$ does not
divide $r$, $rb=qb^2c$ yields the contradiction that $q^{2\alpha
+1}$ divides every entry of $b$.
\end{proof}

The next result mimics Lemma~\ref{preufd}.
\begin{lemma}
\label{prededekind} Let $\Delta$ be a Dedekind domain, and let
$\varrho$ be a proper prime ideal (i.e., maximal ideal) of
$\Delta$. Suppose that $\mathcal{G}(n,\Delta, \varrho)$ has
non-trivial torsion elements. Then
\begin{itemize}
\item[{\rm (i)}] there is a unique prime $p\in \mathbb{Z}$ such
that $p\in \varrho;$ \item[{\rm (ii)}] for some $b\in
\mathrm{Mat}(n,\varrho)$, $pb_{j,k}=-\sum_{i=2}^p{p \choose i}
b_{j,k}^{(i)}$, where $b_{j,k}^{(i)}$ denotes the $(j,k)$th entry
of $b^{i};$ and \item[{\rm (iii)}] every torsion element of
$\mathcal{G}(n,\Delta, \varrho)$ has $p$-power order.
\end{itemize}
\end{lemma}
\begin{proof}
(Cf.~ the proof of \cite[Theorem 4, p.~70]{Suprunenko2}.) If
$h\in \mathcal{G}_{\varrho}$ 
has prime order $p$ then
$$
pb+\cdots + {p \choose i}b^i + \cdots + b^p =0_n
$$
for some $b\in \mathrm{Mat}(n,\varrho)$, reading the binomial
coefficients modulo $\mathrm{char}\, \Delta$. Now (ii) is clear.

Let $l\geq 1$ be the integer such that $b_{j,k}\in \varrho^l$ for
all $j,k$, but $b_{r,s}\not \in \varrho^{l+1}$ for some $r,s$.
(Such an integer $l$ exists, because the
ideal $I$ of $\Delta$ generated by the entries of $b$ is contained
in $\varrho$, $I=\allowbreak \varrho J$ where $J$ is the ideal
$\varrho^{-1}I$, and $J$ has a maximal power of $\varrho$ in its
primary decomposition.) Then (ii) and $b_{j,k}^{(i)} \in
\varrho^{il}$ imply that $pb_{j,k}\in \varrho^{2l}$. Suppose that
$p\not \in \varrho$. Since $\varrho$ is a maximal ideal,
 $\Delta$ is generated by $p$ and $\varrho$.
Let $x \in \Delta,\, y\in\varrho$ be such that $px+y=1$. Then
$b_{j,k}=\allowbreak pb_{j,k}x+\allowbreak b_{j,k}y\in
\varrho^{2l}+\varrho^l\varrho\subseteq \varrho^{l+1}$, a
contradiction. Thus $p\in \varrho$. Moreover, $p$ is the unique
prime integer in $\varrho$ (otherwise $1\in \varrho$
by B\'{e}zout's lemma), so that every torsion element of
$\mathcal{G}_\varrho$ is a $p$-element.
\end{proof}
\begin{proposition}
\label{dedekindinallchar} Let $\Delta$ and $\varrho$ be as in
Lemma{\em ~\ref{prededekind}}.
\begin{itemize}
\item[{\rm (i)}] If $\mathrm{char}\, \Delta=t>0$ then every
torsion element of $\mathcal{G}(n,\Delta, \varrho)$ is a
$t$-element. \item[{\rm (ii)}] If $\mathrm{char}\, \Delta=0$,
$2\not \in \varrho$ and $p\not \in \varrho^2$ for all primes $p\in
\mathbb{Z}$, then $\mathcal{G}(n,\Delta, \varrho)$ is
torsion-free.
\end{itemize}
\end{proposition}
\begin{proof}
(Cf.~\cite[\mbox{p.} 70]{Suprunenko2}.)

(i) \,  This  follows at once from Lemma~\ref{prededekind} (i) and
(iii).

(ii) \,  If $\mathcal{G}_\varrho$ has non-trivial torsion then
$\mathcal{G}_{\varrho}$ has elements of $p$-power order, where
$p\in \varrho$ for an odd prime $p$. By Lemma~\ref{prededekind}
(ii), there exist an element $b$ of $\Delta$ and an integer $l$
such that $b\in \varrho^{l}\setminus \varrho^{l+1}$ (so that
$\varrho^l$ is the largest power of $\varrho$ appearing in the
primary decomposition of the ideal $b\Delta$) and $pb\in
\varrho^{2l+1}$. Certainly then $pb \in \varrho^{l+2}$.

We now derive a contradiction. First, $pb\in \varrho^{l+2}$
implies that $p\Delta \cdot b \Delta\subseteq pb\Delta\subseteq
\varrho^{l+2}$, so $\varrho^{l+2}$ appears in the primary
decomposition of $p\Delta \cdot b \Delta$. But since $p\in \varrho
\setminus \varrho^2$, we know that $\varrho^{l+1}$ is the largest
power of $\varrho$ appearing in this decomposition. 
\end{proof}

To round out this section, we look briefly at how congruence
homomorphisms may be applied in practice to finitely generated
matrix groups. The congruence image should be a matrix
group for which solutions to the specific problems are known (for
example, the image is over a finite field), and the congruence
kernel should be either torsion-free or 
consist of unipotent elements.

Let $\mathbb{F}$ be the field of fractions of the integral domain
$\Delta$, and let $R$ be a finitely generated subring of
$\mathbb{F}$. In particular, if $G=\langle g_1,\ldots , g_r
\rangle\leq \gln$ then $R=R(G)$ denotes the ring generated by the
entries of the elements of  $\{ g_i, g_i^{-1} \mid 1\leq i \leq
r\}$. Obviously $G \leq$ $\mathrm{GL}(n,R(G))$.
 
Let $\pi \subseteq \Delta$ be the set of denominators of the 
generators in a finite generating set of $R$. Denote by
$\Delta_\pi$ the ring of fractions with denominators in the
submonoid of $\Delta^{\times}$ generated by $\pi$ (\cite[\pa
311]{Cohn}). Of course, $R \subseteq \Delta_\pi$. If $\Delta$ is a
UFD or Dedekind domain then $\Delta_\pi$ is a UFD or Dedekind
domain, respectively (\cite[Theorem 3.7, \pa 315]{Cohn} and
\cite[Corollary 5.2, \pa 322]{Cohn}). Since the quotient of a
finitely generated commutative ring by a maximal ideal is a finite
field (\cite[4.1, \pa 50]{Wehrfritz}), if $\Delta$ is finitely
generated and $\varrho$ is a maximal ideal of $\Delta_\pi$ then
$\Delta_\pi/\varrho$ is a finite field. Thus, if $G\leq$ $\gln$
then $\psi_\varrho: \mathrm{GL}(n,\Delta_\pi) \rightarrow
\mathrm{GL}(n,\Delta_\pi/\varrho)$ maps $G$ into some $\glnq$.

We give two examples to illustrate the above that 
are of computational interest.

\begin{example}
\label{ex(a)} {\em  Let $\mathbb{F}$ be an algebraic number field,
and let $\Delta$ be the ring of integers of $\mathbb{F}$. Since
$\Delta$ is finitely generated, $\Delta_\pi$ is a finitely
generated Dedekind domain. Let $\varrho$ be a maximal 
(proper prime) ideal of $\Delta_\pi$ not containing 2, such that $p
\not \in \varrho^2$ for all primes $p\in \mathbb{Z}$; then
$\mathcal{G}(n, \Delta_\pi, \varrho)$ is torsion-free by
Proposition~\ref{dedekindinallchar}, and $\Delta_\pi/\varrho$ is a
finite field. If $\mathbb{F}=\mathbb{Q}$, say, then
$\Delta=\mathbb{Z}$, and if we choose an odd prime
$p\in\mathbb{Z}$ which does not divide any element of $\pi$ then
$\varrho= p\Delta_\pi$ is as required. In this case
$\Delta_\pi/\varrho = \mathrm{GF}(p)$.

For number fields $\mathbb{F}$ in general, to find $\varrho$ we
can reduce to $\mathbb{Q}$ after fixing a $\mathbb{Q}$-basis of
$\mathbb{F}$; this however has the disadvantage of blowing up the
size of matrices. Alternatively we proceed as follows. Suppose
that $\mathbb{F}=\mathbb{Q}(\alpha)$ contains all generators of
$R$, where $\alpha$ is an algebraic integer. Let $m$ be the degree
of the minimal polynomial of $\alpha$. Expressing each generator
of $R$ uniquely as a $\mathbb{Q}$-linear combination of $\{ 1,
\alpha, \ldots , \alpha^{m-1}\}$, and thereafter obtaining each
generator in the form $\beta/z$ where $\beta$ is an algebraic
integer and $z\in \mathbb{Z}$, we can find $\pi \subseteq$
$\mathbb{Z}$. If $p\in \mathbb{Z}$ is an odd prime element of
$\Delta$ not dividing any element of $\pi$ then
$\varrho=p\Delta_\pi$ is a maximal ideal of $\Delta_\pi$ such that
$\Delta_\pi/\varrho = \mathrm{GF}(p^l)$ for some $l\leq m$, and
$\mathcal{G}(n,\Delta,\varrho)$ is torsion-free. Note that 
$R\subseteq\mathbb{Z}_{\pi}[\alpha ]$. So the reduction mod $p$ congruence
homomorphism on a finitely generated subgroup $G$ of $\gln$ with
$R=R(G)$ is easily described. To evaluate $\psi_\varrho$,
we reduce elements of $\mathbb{Z}_\pi$ mod $p$, and if $f(X)\in
\mathbb{Z}[X]$ is the minimal polynomial of $\alpha$ then
$\psi_\varrho(\alpha)$ is a root of the mod $p$-reduction
$\bar{f}(X)$ of $f(X)$. If $\bar{f}(X)$ is irreducible over
$\mathrm{GF}(p)$ then $l=m$; otherwise $l<m$.}
\end{example}
\begin{example}
\label{ex(b)} {\em  Let $\mathbb{F}$ be a function field $P(X)$,
and let $\Delta$ be the polynomial ring $P[X]$, where $P$ is a
UFD. Then $\Delta$ is a UFD (\cite[\pa 316]{Cohn}), and therefore
so too is $\Delta_\pi$. Let $q=X-\alpha$, where $\alpha$ is not a
root of any element of $\pi$. (If $P$ is infinite then of course
$\alpha$ always exists in $P$; else we can replace the finite
field $P$ by a finite extension containing $\alpha$.) Then
$\varrho = q \Delta_\pi$ is a prime ideal of $\Delta_\pi$. By
Proposition~\ref{ufdcase}, either
$\mathcal{G}(n,\Delta_\pi,\varrho)$ is torsion-free, or every
torsion element of $\mathcal{G}(n,\Delta_\pi, \varrho)$ is
unipotent. The effect of $\psi_\varrho$ is just substitution of
$\alpha$ for the indeterminate $X$ in elements of $\Delta_\pi$.
Hence $\psi_\varrho (\Delta_\pi)$ can be viewed as a subring of
$P$. If $P$ is finite then $\psi_\varrho (\Delta_\pi)$ is also a
finite field. When $P$ has characteristic zero we apply a
suitable congruence homomorphism over the finitely generated
integral domain $\psi_\varrho (\Delta_\pi)\subseteq P$, in line
with the following simple observation: if $\psi_{\varrho_1}:
\Delta \rightarrow \Delta/\varrho_1$ and
$\psi_{\varrho_2/\varrho_1}:\Delta/\varrho_1 \rightarrow
\Delta/\varrho_2$ are (natural) homomorphisms of integral domains
such that $\mathcal{G}(n,\Delta, \varrho_1)$ and $\mathcal{G}(n,
\Delta/\varrho_1, \varrho_2/\varrho_1)$ are both torsion-free,
then $\mathcal{G}(n,\Delta, \varrho_2)$ is torsion-free. 
Say $P=\mathbb{Q}$; then $\psi_\varrho (
\Delta_\pi)\subseteq \mathbb{Z}_{\pi_1}$ for some finite subset
$\pi_1$ of $\mathbb{Z}\setminus \{ 0 \}$, and we are back to the
situation of Example~\ref{ex(a)}.}
\end{example}

\section{Computing with nilpotent matrix groups}

In this section we design algorithms for
computing with matrix groups over a field $\mathbb{F}$, as set out
in the introduction. We are guided 
by the
algorithms and results in \cite{Large}. Although only finite
$\mathbb{F}$ were treated in \cite{Large}, most of that
paper's fundamental results are valid over any $\mathbb{F}$.

\subsection{Splitting nilpotent linear groups}
\label{splitsect} In linear group theory, we
often reduce problems to the completely reducible case. This
reduction is more straightforward for nilpotent linear groups than
it is for arbitrary linear groups (see, e.g.,
\cite[Subsection~2.1]{MR2232095}). In this subsection we consider
a computational approach to the reduction.

Our starting point is the Jordan decomposition. Recall that $h\in
\gln$ is \emph{diagonalizable} if $h$ is conjugate to a
diagonal matrix over some extension of $\mathbb{F}$, and $h$
is \emph{semisimple} if $\langle h\rangle \leq \gln$ is completely
reducible. A semisimple element of $\gln$ need not be
diagonalizable, unless $\mathbb{F}$ is perfect: then the two
concepts coincide. Denote the algebraic closure of $\mathbb{F}$ by
$\overline{\mathbb{F}}$. For each $g \in\allowbreak \gln$, there is a unique
unipotent matrix $g_u\in\mathrm{GL}(n,\overline{\mathbb{F}}
\hspace*{.1mm} )$ and a unique diagonalizable matrix $g_s\in
\mathrm{GL}(n,\overline{\mathbb{F}} \hspace*{.1mm} )$ such that
$g=$ $g_sg_u=g_ug_s$ (see \cite[7.2, p.~91]{Wehrfritz}).
Note that this Jordan decomposition of $g$ 
is the same over 
every extension of
$\mathbb{F}$. If $\mathbb{F}$ is perfect then by \cite[Proposition~1, 
p.~134]{Segal}, $g_u$ and $g_s$ are in $\gln$.

An algorithm to compute the Jordan decomposition can be found in
\cite[Appendix A]{Babaietal}. Systems such as {\sf GAP} also
have standard functions for computing the decomposition.

Let $G=\langle
g_1, \ldots, g_r\rangle \leq \gln$.  Define
$$
G_u=\langle (g_1)_u, \ldots, (g_r)_u\rangle \qquad \mbox{and}
\qquad G_{s} = \langle (g_1)_s, \ldots, (g_r)_s\rangle .
$$
Since $g_i=$ $(g_i)_u (g_i)_s \in \langle G_u,G_{s} \rangle$,
clearly $G\leq G^*:= \langle G_u, G_{s}\rangle$. In general,
neither $G_u$ nor $G_s$ are necessarily subgroups of $G$.
\begin{lemma}
\label{corsix}
\begin{itemize}
\item[{\rm (i)}] $G$ is nilpotent if and only if $G_u$, $G_{s}$
are nilpotent and $[G_u, G_{s} ]=1$. \item[{\rm (ii)}] If $G$ is
nilpotent then $G\leq G^*= G_u\times G_s$.
\end{itemize}
\end{lemma}
\begin{proof}
If $G$ is nilpotent then the assignments $g \mapsto g_u$ and $g
\mapsto g_s$ define homomorphisms $G\rightarrow$ $G_u$ and $G
\rightarrow G_s$; furthermore $G^*=G_u \times G_s$ (see
\cite[Proposition 3, \pa 136]{Segal}). On the other hand, if
$G_u$, $G_{s}$ are nilpotent and $[G_u, G_{\!s}]=1$, then $G^*$
and thus $G\leq G^*$ are nilpotent.
\end{proof}
\begin{remark}
Let $G$ be nilpotent. Then $G_u=\{ g_u \mid g\in G\}$ and $G_s=\{
g_s \mid g\in G\}$. Sometimes $G=G_u\times G_s$. For
example, this is true if $\mathbb{F}$ is finite. As another
example, if $G$ is an algebraic group (over algebraically closed
$\mathbb{F}$) then $g_u, g_s\in G$ for all $g\in G$, so that $G=
G^*$.
\end{remark}
\begin{lemma}
\label{lopez} $G_u$ is nilpotent if and only if it is unipotent,
i.e., conjugate to a subgroup of the group
$\mathrm{UT}(n,\mathbb{F})$ of all upper unitriangular matrices
over $\mathbb{F}$.
\end{lemma}
\begin{proof}
A unipotent group is unitriagularizable (see \cite[1.21, \mbox{p.}
14]{Wehrfritz}). If $G_u$ is nilpotent then $G_u=\{g_u\mid g\in
G\}$ is unipotent. As is well-known, $\mathrm{UT}(n,\mathbb{F})$
is nilpotent, of class $n-1$.
\end{proof}

In \cite[Subsection 2.1]{Large}, a recursive procedure is given
for deciding whether a group generated by unipotent matrices (over
any field $\mathbb{F}$) is unipotent. We label that procedure
${\tt IsUnipotent}$ here.

\bigskip

\medskip

\hspace*{-1.5mm} ${\tt IsUnipotent}(H)$

\vspace*{1.5mm}

Input: $H=\langle h_1, \ldots , h_r\rangle$, $h_i\in
\mathrm{GL}(n,\mathbb{F})$ unipotent, $\mathbb{F}$ any field.

Output: a $\mathrm{UT}(n,\mathbb{F})$-representation of $H$, or a
message `false' meaning that $H$ is not unipotent.

\bigskip

\medskip

Lemmas~\ref{corsix} and \ref{lopez}, and ${\tt IsUnipotent}$,
equate nilpotency testing of $G\leq \gln$ to testing nilpotency of
$G_{s}$ and testing whether $[G_u,G_{s}]=1$.

If $G_u$ is unipotent then ${\tt IsUnipotent}$ finds a
$\mathrm{UT}(n,\mathbb{F})$-representation of $G_u$ by
constructing a series
\begin{equation}
\label{sseries}
 V=V_0>V_1> \cdots > V_{l-1}>V_l=0
\end{equation}
of $G_u$-submodules of the underlying space $V$ for $\gln$, such
that $G_u$ acts trivially on each factor $V_{i-1}/V_i$. In fact,
$V_{i-1}/V_i$ is the fixed point space
$\mathrm{Fix}_{G_u}(V/V_i)$. We get more when $G$ is nilpotent, by
the next two lemmas.
\begin{lemma}
\label{precoolfacts} Each unipotent element of a completely
reducible nilpotent subgroup of $\mathrm{GL}(n,\mathbb{F})$ is
trivial.
\end{lemma}
\begin{proof}
This follows from \cite[Corollary 1, \pa 239]{Suprunenko2}.
\end{proof}
\begin{lemma}
\label{coolfacts}
 Let $G\leq \gln$ be nilpotent, $\mathbb{F}$ a perfect field. Then
\begin{itemize}
\item[{\rm (i)}] $G_s$ is completely reducible over $\mathbb{F};$
\item[{\rm (ii)}]  $G$ is completely reducible over $\mathbb{F}$
if and only if $G_u=1$.
\end{itemize}
\end{lemma}
\begin{proof}
 A solvable group of diagonalizable matrices (over any
field) is completely reducible by \cite[Theorem 5, \mbox{p.}
172]{Suprunenko2}. Since $G_s\leq \gln$ consists entirely of
diagonalizable matrices, if $G_u=1$ then $G=G_s$ is completely
reducible. The converse is Lemma~\ref{precoolfacts}.
\end{proof}

If $G$ is nilpotent and $\mathbb{F}$ is perfect then
Lemmas~\ref{corsix} and ~\ref{coolfacts} imply that each factor
$V_{i-1}/V_i$ of \eqref{sseries} is a completely
reducible $G^*$-module. As a subgroup of a nilpotent completely
reducible subgroup of $\gln$ is completely reducible by
\cite[Theorem 5, \mbox{p.} 239]{Suprunenko2}, we see that if
$G$ is nilpotent then ${\tt IsUnipotent}$ constructs completely
reducible modules not just for $G^*$ but also for $G$.

We now give a procedure for reducing nilpotency testing of $G\leq
\gln$ to testing nilpotency of a matrix group generated by
diagonalizable matrices.

\bigskip

\medskip

${\tt Reduction}(G)$

\vspace*{1.5mm}

Input: $G=\langle g_1, \ldots, g_r\rangle\leq \gln$, $\mathbb{F}$
any field.

Output: $G_{s}$, a $\mathrm{UT}(n,\mathbb{F})$-representation of
$G_u$, and a message that $[G_u,G_s]=1$; or a message

`false' meaning that $G$ is not nilpotent.

\vspace*{1.5mm}

for $i\in \{1, \ldots, r\}$ do

\hspace*{2mm} find $(g_i)_u$, $(g_i)_s$;

$G_u:=\langle (g_i)_u : 1 \leq i \leq r \rangle$, $G_{s} :=\langle
(g_i)_s : 1 \leq i \leq r \rangle$;

if ${\tt IsUnipotent}(G_u)=$ `false'

\hspace*{2mm} then return `false';

else $N:= [G_u,G_{s} ]$;

if $N\neq 1$

\hspace*{2mm} then return `false';

else return $G_{s}$.

\bigskip

\medskip

There are other reductions to the completely reducible case. For
example, we could compute the radical $R$ of the
enveloping algebra $\langle G \rangle_{\mathbb{F}}$, and then the
radical series
$$
V\supset RV \supset R^2V \supset \cdots \supset R^{m}V=0
$$
(see \cite{Ronyai} for methods to compute $R$). Each
term $R^iV$ in this series is a $G$-module, and each factor
$R^iV/R^{i+1}V$ is a completely reducible $G$-module. 
The radical series may be used to write $G$ in block upper triangular form,
thereby obtaining a homomorphism $\theta$ of $G$ onto a completely
reducible subgroup of $\gln$. If $G$ is nilpotent then $\ker
\theta$ is the unipotent radical of $G$ (the unique maximal
unipotent normal subgroup of $G$), and $\ker \theta$ commutes with
every diagonalizable element of $G$ (see \cite[7.11, \mbox{p.}
97]{Wehrfritz}).

\subsection{Further background} \label{firststeps}

We now prepare the way  for applying
Section~\ref{changingviacongh} to nilpotency testing over an
arbitrary field $\mathbb{F}$.

As usual, $Z_i(G)$ will denote the $i$th term of the upper 
central series of $G$; i.e., $Z_0(G)=1$ and 
$Z_i(G)/Z_{i-1}(G)={\sf Z}(G/Z_{i-1}(G))$.
\begin{lemma}
\label{dixoneasy} If $G$ is a completely reducible nilpotent
subgroup of $\mathrm{GL}(n,\mathbb{F})$ then $|G:\mathsf{Z}(G)|$
is finite.
\end{lemma}
\begin{proof}
See \cite[Corollary 6.5, \pa 114]{Dixon}, or \cite[Theorem 1,
\mbox{p.} 208]{Suprunenko2}. Also cf.~Zassenhaus' result
\cite[3.4, \pa 44]{Wehrfritz}.
\end{proof}
\begin{remark}
\label{tooexplicit?} \cite[Theorem 1, \mbox{p.} 208]{Suprunenko2}
is stated for irreducible groups only. The result for completely
reducible nilpotent subgroups $G$ of $\mathrm{GL}(n,\mathbb{F})$
follows from this, because $G/{\sf Z}(G)$ is isomorphic to a
subgroup of the direct product of central quotients of irreducible
nilpotent linear groups (each of degree no more than $n$).
\end{remark}
\begin{lemma}
\label{bigun} Let $G$ be a completely reducible nilpotent subgroup
of $\mathrm{GL}(n,\mathbb{F})$. If $N$ is a torsion-free normal
subgroup of $G$ then $N\leq \mathsf{Z}(G)$.
\end{lemma}
\begin{proof}
Suppose that $N\not\subseteq \mathsf{Z}(G)$. Then $N{\sf
Z}(G)/{\sf Z}(G)$ is a non-trivial normal subgroup of the
nilpotent group $G/{\sf Z}(G)$, so has non-trivial intersection
with $Z_2(G)/{\sf Z}(G)$. Let $x\in N\cap Z_2(G)$, $x\not \in {\sf
Z}(G)$. By Lemma~\ref{dixoneasy}, $x^m\in {\sf Z}(G)$ for some
$m$. Choose $g \in G$ such that $x^g = x\epsilon$ for some
$\epsilon \in {\sf Z}(G)$, $\epsilon \neq 1$. Then $x^m= (x^m)^g=
(x^g)^m=x^m\epsilon^m$ implies that $\epsilon$ is a non-trivial
torsion element of $G$. But $\epsilon=x^{-1}x^g\in N$. 
\end{proof}

Now let $G$ be a finitely generated subgroup of $\gln$. Suppose
that $\Delta$ is a finitely generated subring of $\mathbb{F}$ such
that $G\leq$ $\mathrm{GL}(n,\Delta)$, and let $\varrho$ be an
ideal of $\Delta$. We continue with the notation
$\mathcal{G}(n,\Delta,\varrho)$ and $G_\varrho$ adopted in
Section~\ref{changingviacongh} for congruence subgroups. Without
loss of generality, we may assume that $\mathbb{F}$ is the field
of fractions of $\Delta$.
\begin{lemma}
\label{biguncor} Suppose that $\mathcal{G}(n,\Delta,\varrho)$ is
torsion-free if char \hspace*{-1.5mm} $\Delta=0$, and  all torsion
elements of $\mathcal{G}(n,\Delta,\varrho)$ are  unipotent if char
\hspace*{-1.5mm} $\Delta > 0$. Let $G$ be completely reducible as
a subgroup of $\mathrm{GL}(n,\mathbb{F})$. If $G$ is nilpotent
then $G_\varrho$ is a torsion-free central subgroup of $G$.
\end{lemma}
\begin{proof}
By Lemma~\ref{precoolfacts} and the hypotheses, $G_\varrho$ is
torsion-free. Hence the result follows from Lemma~\ref{bigun}.
\end{proof}

Examples~\ref{ex(a)} and \ref{ex(b)}
 show how to select $\varrho$ as stipulated 
in Lemma~\ref{biguncor}
for various $\mathbb{F}$ and $\Delta$. Also, 
Subsection~\ref{splitsect} shows how to split off a completely
reducible subgroup of $\gln$ from an arbitrary finitely generated
nilpotent subgroup of $\gln$.
\begin{theorem}
\label{humus} Suppose that $\mathcal{G}(n,\Delta,\varrho)$ is as
in Lemma{\em ~\ref{biguncor}}, and that $G$ is completely reducible.
Then $G$ is nilpotent if and only if $\psi_\varrho(G)$ is
nilpotent and $G_\varrho\leq {\sf Z}(G)$.
\end{theorem}
\begin{proof}
One direction is elementary, the other is
 Lemma~\ref{biguncor}.
\end{proof}

Theorem~\ref{humus} transforms nilpotency testing of a finitely
generated completely reducible subgroup $G$ of $\gln$ into an
equivalent pair of problems: testing whether $G_\varrho \leq {\sf
Z}(G)$, and testing whether $\psi_\varrho(G)$ is nilpotent. If
$\varrho$ is a maximal ideal then $\Delta/\varrho$ is
a finite field, and we can test nilpotency of $\psi_\varrho(G)$ as
in \cite{Large}. To test whether $G_\varrho\leq {\sf Z}(G)$ we
need a generating set for $G_\varrho$. This may be achieved if
together with the input generating set $\{ g_1, \ldots , g_r\}$ for
$G$, we know either (i) a transversal for the cosets of
$G_\varrho$ in $G$, or (ii) a presentation of $G/G_\varrho\cong
\psi_\varrho(G)$. In case (i), as long as the index
$|G:G_\varrho|=|\psi_\varrho(G)|$ is not too large then the
Schreier method \cite[Section~2.5, pp.~41-45]{CGThandbook}
is a realistic option for finding a generating set of $G_\varrho$.
In case (ii), suppose that every relator $w_j$ in the known
presentation of $\psi_\varrho(G)$ is a word in the
$\psi_\varrho(g_i)$. Then by replacing each occurrence of
$\psi_\varrho(g_i)$ in $w_j$ by $g_i$, $1 \leq i \leq r$, we get a
generating set for a subgroup of $G$ whose normal closure 
is $G_\varrho$. (This is the `normal subgroup generators' method;
cf.~\cite[\mbox{pp}. 299-300]{CGThandbook}.) As a consequence, the
following lemma solves the problem of testing whether $G_\varrho$
is central in $G$.
\begin{lemma}
\label{easynormsgpgens} Let $G=\langle g_1, \ldots ,
g_r\rangle\leq \mathrm{GL}(n,\mathbb{F})$ and
\[
\psi_\varrho(G)= \langle \psi_\varrho(g_1), \ldots ,
\psi_\varrho(g_r) \mid w_1(\psi_\varrho(g_i)), \ldots ,
w_s(\psi_\varrho(g_i)) \rangle .
\]
Then $G_\varrho$ is the normal closure in $G$ of the subgroup
\[
\widetilde{G_\varrho}= \langle w_1(g_i), \ldots , w_s(g_i) \rangle .
\]
Hence $G_\varrho \leq {\sf Z}(G)$ if and only if $w_j(g_i)\in {\sf
Z}(G)$ for all $j$, $1\leq j\leq s$, in which case $G_\varrho =
\widetilde{G_\varrho}$.
\end{lemma}

\subsection{Deciding finiteness}
\label{decfinite} After testing nilpotency of $G\leq \gln$, we can
move on to tackle other basic computational problems for $G$,
such as testing whether $G$ is finite.

Deciding finiteness of matrix groups over algebraic number fields
and functional fields has been considered by various authors, and a
practical implementation was written by Beals in {\sf GAP} for
groups over $\mathbb{Q}$ (see \cite{DBLP:conf/issac/BabaiBR93}).
The method for deciding finiteness that we introduce in this
subsection is a general approach to the problem that is uniform
with respect to the ground field. We apply it here only for
nilpotent groups, while the general case is part of separate
research.

Let $\mathrm{char} \, \mathbb{F}=0$,  and let $\varrho$ be an
ideal of the subring $\Delta$ of $\mathbb{F}$ such that
$\Delta/\varrho$ is finite. Suppose that
$\mathcal{G}(n,\Delta,\varrho)$ is torsion-free. Then, obviously,
$G\leq \mathrm{GL}(n,\Delta)$ is finite if and only if $G_\varrho$
is trivial. This suggests a very simple and general finiteness
test for $G$. However, efficiency of this test depends on knowing
an efficient method to decide whether $G_\varrho$ is trivial. If
$G$ is nilpotent then we have such a method by
Lemma~\ref{easynormsgpgens}.

\bigskip

\medskip

 \hspace*{-1.5mm} ${\tt IsNilpotentFinite}(G)$

\vspace*{1.5mm}

Input: A nilpotent subgroup $G=\langle g_1, \ldots , g_r\rangle$
of $\mathrm{GL}(n,\mathbb{F})$, $\mathrm{char}\, \mathbb{F}=0$.

Output: a message `true' meaning that $G$ is finite; and `false'
otherwise.

\vspace*{1.5mm}

if $G_u \neq 1$

\hspace*{2mm} then return `false';

if $G_\varrho\neq 1$

\hspace*{2mm} then return `false';

else return `true'.

\bigskip

\medskip

Once $G$ is confirmed to be finite, then we know that
$|G|=|\psi_\varrho(G)|$. Computing the order of $G$ is thus
reduced to the order problem for matrix groups over a finite
field. Testing whether $G_\varrho\neq 1$ in ${\tt
IsNilpotentFinite}(G)$ is feasible by Lemma~\ref{easynormsgpgens},
because we can compute a presentation of
the nilpotent group $\psi_\varrho(G)$ over a finite field without
difficulty.

Now we consider $\mathbb{F}$ of positive characteristic.
\begin{lemma}
\label{schurfurst} Let $G=\langle g_1, \ldots , g_r\rangle$ be a
nilpotent subgroup of $\gln$, $\mathrm{char}\, \mathbb{F}>0$. Then
$G_u$ is finite.
\end{lemma}
\begin{proof}
Schur's First Theorem \cite[\mbox{p.} 181]{Suprunenko2} asserts
that a periodic subgroup of $\gln$ is locally finite. As $G$ is
nilpotent, $G_u$ is unipotent and so periodic. Then the result
follows, because $G_u=$ $\langle (g_1)_u, \ldots , (g_r)_u\rangle$
is finitely generated.
\end{proof}

By Lemmas~\ref{corsix} and \ref{schurfurst}, if $G$ is nilpotent
then $G$ is finite precisely when $G_s$ is finite. Let
$\mathbb{F}$ be perfect, and suppose that all torsion elements of
$\mathcal{G}(n,\Delta,\varrho)$ are unipotent. Then as
$G_s/(G_s)_\varrho$ is finite, $G$ is finite if and only if
$(G_s)_\varrho$ is trivial, by Lemma~\ref{biguncor}. If
$\mathbb{F}$ is not perfect then we can still find a normal
unipotent subgroup $U$ of $G$ such that $G/U$ is isomorphic to a
completely reducible subgroup of $\gln$ (see the
end of Subsection~\ref{splitsect}), and the above reasoning goes
through.

For another method to decide finiteness of $G$, 
incorporated with nilpotency testing of $G$, see
Subsection~\ref{firstuseofabseries} below.

\subsection{Polycyclic presentations}

A finitely generated nilpotent group is polycyclic, and therefore
has a (consistent) polycyclic presentation. One benefit of
possessing a polycyclic presentation for a nilpotent subgroup $G$
of $\gln$ is that we gain access to the numerous exisiting
algorithms for abstract polycyclic groups (see \cite[Chapter
9]{Sims}, \cite[Chapter 8]{CGThandbook}, and the {\sf GAP} package
`Polycyclic'~\cite{PolycycPackage}), which may be used to
further investigate the structure of $G$.

The papers \cite{BettinaAssmann,BettinaAssmann2} deal with the
problem of constructing a polycyclic presentation for a finitely
generated subgroup $G$ of $\glq$. Specifically, the algorithm
${\tt PolycyclicPresentation}(G)$ in \cite{BettinaAssmann}
attempts to compute polycyclic presentations for
$\psi_\varrho(G)$, $G_\varrho/U_\varrho$, and $U_\varrho$, where
$\varrho = p\mathbb{Z}_{\pi}$ for a finite set $\pi$ of primes not
containing the odd prime $p$,
$\psi_\varrho:\mathrm{GL}(n,\mathbb{Z}_\pi)\rightarrow\mathrm{GL}(n,p)$
is the associated congruence homomorphism, and $U_\varrho$ is a
unipotent radical of $G_{\varrho}$. If $G$ is polycyclic then the
algorithm returns a polycyclic presentation of $G$. The algorithm
fails to terminate if $G$ is solvable but not polycyclic
(i.e., $U_\varrho$ is not finitely generated). In this
subsection we propose a modification of ${\tt
PolycyclicPresentation}$ which either returns a polycyclic
presentation of $G$, or detects that $G$ is not nilpotent.

The paper \cite{BettinaAssmann2} contains another algorithm, ${\tt
IsPolycyclic}(G)$, for polycyclicity testing of a subgroup $G$ of
$\glq$. ${\tt IsPolycyclic}(G)$ always terminates, returning
either a polycyclic presentation of $G$, or a message that $G$ is
not polycyclic. A nilpotency testing algorithm based on ${\tt
IsPolycyclic}(G)$ is also given in \cite{BettinaAssmann2}. That
algorithm has the following stages: (i) testing whether $G$ is
polycyclic, (ii) testing whether $G/U$ is nilpotent, where $U$ is
a unipotent radical of $G$, and (iii) testing whether $G$ acts
nilpotently on $U$. Our approach in this subsection avoids the
possibly time-consuming step (i), and replaces step (iii) with a
simpler test.

The strategy of our algorithm is as follows. Let $G$ be a finitely
generated subgroup of $\gln$, where for convenience $\mathbb{F}$
is assumed to be perfect. After applying ${\tt Reduction}(G)$, we
will know either that $G$ is not nilpotent, or that $G\leq \langle
G_u , G_s\rangle$, $G_u$ is unipotent, and $[G_u,G_s]=1$. In the
latter event, we find polycyclic
presentations of $G_u$ and $G_s$. Note that if we proceed
further after ${\tt Reduction}(G)$, then the finitely generated
nilpotent group $G_u \leq \mathrm{UT}(n,\mathbb{F})$ is definitely
polycyclic. Next, we find presentations of 
$\psi_\varrho(G_s)$ and $(G_s)_\varrho$ for a suitable 
$\varrho$. 
Recall that if $G$
is nilpotent then $(G_s)_\varrho \leq {\sf Z}(G_s)$ is abelian
(see Lemma~\ref{biguncor}).

\bigskip

\medskip

${\tt PresentationNilpotent}(G)$

\vspace*{1.5mm}

Input: $G=\langle g_1, \ldots,  g_r\rangle\leq \gln$, $\mathbb{F}$
perfect.

Output: a polycyclic presentation of $G$, or a message `false'
meaning that $G$ is not nilpotent.

\vspace*{1.5mm}

\begin{enumerate}

\item If ${\tt Reduction}(G)=$ `false' then return `false'; else
go to step (\ref{notwo}). \item \label{notwo} Determine a
polycyclic presentation of $G_u$ as a finitely generated subgroup
of $\mathrm{UT}(n,R)$, $R$ a finitely generated subring of
$\mathbb{F}$. \item \label{pcpimage} Compute a generating set for
$\psi_\varrho(G_s)$, and use this to attempt to construct a
polycyclic presentation of $\psi_\varrho(G_s)$. Return `false' if
the attempt fails. \item \label{pcpkernel} Determine a generating
set for $(G_s)_\varrho$. If $(G_s)_\varrho$ is not central in
$G_s$ then return `false'. Else construct a polycyclic
presentation of the finitely generated abelian group
$(G_s)_\varrho$. \item \label{almostthere} Combine the
presentations of $\psi_\varrho(G_s)$ and $(G_s)_\varrho$ found in
steps (\ref{pcpimage}) and (\ref{pcpkernel}) to get a polycyclic
presentation of $G_s$. \item \label{finallythere} Combine the
presentations of $G_u$ and $G_s$ found in steps (\ref{notwo}) and
(\ref{almostthere}) to get a polycyclic presentation of $G^*= G_u
G_s$ and thence a polycyclic presentation of $G\leq G^*$.
\end{enumerate}

\bigskip

\medskip

Implementation of ${\tt PresentationNilpotent}$ depends on 
having algorithms for computing the polycyclic
presentations in steps \eqref{notwo} and \eqref{pcpkernel}. Such
algorithms are available for finite fields and number
fields (see \cite{BettinaAssmann,BettinaAssmann2}).

\subsection{Testing nilpotency using an abelian series; the
adjoint representation} \label{firstuseofabseries}

Methods for testing nilpotency of matrix groups, relying on
properties of nilpotent linear groups, were proposed in
\cite{Large}. Although those methods were applied only to groups
over finite fields, they are valid over other fields as well. In
this subsection we justify this statement.

As in \cite[Subsection 2.2]{Large} we define a recursive procedure
${\tt SecondCentralElement}(G,H)$ which accepts as input finitely
generated subgroups $G, H$ of $\gln$, $\mathbb{F}$ any field,
where $H$ is a non-abelian normal subgroup of $G$. If $G$ is
nilpotent then the recursion terminates in a number of rounds no
greater than the nilpotency class of $G$, returning an element of
$Z_2(H)\setminus {\sf Z}(H)$. We therefore seek an upper bound on
nilpotency class of nilpotent subgroups of $\gln$. (Such a bound
exists only for certain fields $\mathbb{F}$. For instance if
$\mathbb{F}$ is algebraically closed then $\gln$ contains
nilpotent groups of every class; see \cite[Corollary 1, \mbox{p.}
214]{Suprunenko2}.) 
\begin{lemma}
\label{classbdprecursor} Let $G$ be a nilpotent completely
reducible subgroup of $\gln$ contained in $\mathrm{GL}(n,\Delta)$,
$\Delta$ a finitely generated subring of $\mathbb{F}$. Let
$\varrho$ be an ideal of $\Delta$ as in Lemma{\em ~\ref{biguncor}}. Then
the nilpotency class of $G$ is at most the nilpotency class of
$\psi_\varrho(G)$ plus $1$.
\end{lemma}
\begin{proof}
This is clear by Lemma~\ref{biguncor}.
\end{proof}
\begin{example}
{\em Theorem 2 of \cite{MR1878590} gives an upper bound $3n/2$
on the nilpotency class of subgroups of $\glq$. This further
implies an upper bound $3mn/2$ for subgroups of
$\mathrm{GL}(n,\mathbb{P})$, where $\mathbb{P}$ is a number field
of degree $m$ over $\mathbb{Q}$. Suppose that $G$ is a finitely
generated nilpotent subgroup of $\mathrm{GL}(n, \mathbb{Q}(X))$.
Since $\mathrm{UT}(n,\mathbb{F})$ has nilpotency class $n-1$ (see
\cite[Theorem 13.5, \mbox{p.} 89]{Suprunenko2}), it follows from
Lemmas~\ref{corsix} and \ref{classbdprecursor} that the nilpotency
class of $G$ is at most $\frac{3n}{2}+1$. Similar remarks pertain to
groups over $\mathbb{Q}(X_1, \ldots , X_m)$. }
\end{example}
\begin{example}
\label{finitefieldclassbd}
 {\em Let $q$ be a power of a prime $p$.
If $G$ is a finitely generated nilpotent subgroup of $\gln$ for
$\mathbb{F}=\mathrm{GF}(q)(X)$ then $G$ has nilpotency class at
most $l_{n,q}+1$, where $l_{n,q}$ is an upper bound on the class
of nilpotent subgroups of $\glnq$. A formula for $l_{n,q}$ may be
deduced from \cite[Theorem C.3]{Bialo}:
\begin{equation}
l_{n,q} = n  \cdot  \mathrm{max} \{ (t-1)s+1 \mid t \neq p \ \,
\mathrm{prime}, \ t\leq n , \ t^s \ \mathrm{dividing} \
q-1\}\label{classbd}
\end{equation}
That is, $n((t-1)s+1)$ is an upper bound on the class of a Sylow
$t$-subgroup of $\glnq$, where $t^s$ is the largest power of the
prime $t$ dividing $q-1$ (slightly better bounds are known for
special cases, e.g., $t=2$). We restrict to $t\leq n$ in
\eqref{classbd} because a $t$-subgroup of $\glnq$ is abelian if
$t\neq p$ and $t>n$ (cf.~\cite[Lemma 2.25]{Large}).}
\end{example}

We assume henceforth that we are able to specify a number
$k_\mathbb{F}$ such that if termination does not happen in 
$k_\mathbb{F}$ rounds or less then ${\tt
SecondCentralElement}(G,H)$ reports that $G$ is not nilpotent;
otherwise, the procedure returns an element $a \in Z_2(H)\setminus
{\sf Z}(H)$ such that $[G,a] \leq {\sf Z}(H)$.

Other procedures in \cite{Large} that were designed for
finite fields $\mathbb{F}$ also carry over to any $\mathbb{F}$.
Given $a \in Z_2(G)\setminus{\sf Z}(G)$, let $\varphi_a: G
\rightarrow {\sf Z}(G)\cap [G,G]$ be the homomorphism defined by
$g\mapsto [g,a]$. If $G$ is completely reducible then ${\tt
NonCentralAbelian}(G,a)$ returns the abelian normal subgroup
$A=\langle a \rangle^G=$ $\langle a, \varphi_a(G)\rangle$ of $G$,
and ${\tt Centralizer}(G,A)$ returns a generating set for the
kernel ${\sf C}_G(A)$ of $\varphi_a$. ${\tt
NonCentralAbelian}(G,a)$ requires a `cutting procedure' for 
$\langle A \rangle_\mathbb{F}$, to reduce
computations to the case of cyclic $\varphi_a(G)$. Also,
\cite[Lemma 2.17 and Corollary 2.18]{Large} hold for any field
$\mathbb{F}$; hence, as in the finite field case, we get a
moderate upper bound on the index $|G:{\sf C}_G(A)|$.

The cutting procedure described in \cite[Section 3]{Ronyai} finds
the simple components of a finite-dimensional commutative
semisimple algebra over any field $\mathbb{F}$, input by a set of
algebra generators. When $\mathbb{F}=\mathbb{Q}$ another method,
based on \cite[Lemma 5]{MR780623}, can be applied (see
\cite[Section
 5.2]{BettinaAssmann}). The main ingredient here
is an efficient method for factorizing polynomials over
$\mathbb{F}$.

The discussion above shows that the recursive procedure ${\tt
TestSeries}$ of \cite[Subsection 2.4]{Large} can be defined over
any field $\mathbb{F}$. The basic steps in the recursion are
outlined in \cite[\mbox{pp.} 113-114]{Large}. If it does not
detect that an input finitely generated subgroup $G$ of $\gln$ is
not nilpotent, then ${\tt TestSeries}(G,l)$ returns a series
\begin{equation}
\label{chain4}
 \langle 1_n\rangle  \lhd A_1 \lhd A_2 \lhd \cdots \lhd A_l \unlhd
 C_l \lhd  \cdots \lhd C_2 \lhd C_1 \lhd  G
\end{equation}
where the $A_i$ are abelian,
$C_i\unlhd G$ is the centralizer of
$A_i$ in $C_{i-1}$, and the factors $C_{i-1}/C_i$ are abelian.
That is, all factors of consecutive terms in \eqref{chain4} are
abelian, except possibly the middle factor $C_l/A_l$. The
construction of further terms in \eqref{chain4} continues, with
strict inclusions everywhere except possibly in the middle of the
series, as long as $C_l$ is non-abelian.
\begin{lemma}
\label{selfcent} For some $l\leq n-1$, the term $C_l$ in
\eqref{chain4} is abelian.
\end{lemma}
\begin{proof}
Cf.~the proof of \cite[Lemma 2.20]{Large}.
\end{proof}

{\bf N.B.} {\em Until further notice in this subsection, $G$ is
completely reducible.}
\begin{lemma}
\label{wherez} ${\sf Z}(G)$ is contained in every term $C_i$ of
 $\eqref{chain4}$. Therefore, if $G$ is nilpotent then
$G/C_l$ is finite.
\end{lemma}
\begin{proof}
Certainly ${\sf Z}(G) \leq C_1={\sf C}_G(A)$. Assume  that ${\sf
Z}(G) \leq C_{k-1}$; then ${\sf Z}(G)$ is contained in the
$C_{k-1}$-centralizer $C_k$ of $A_k$. The second statement is now
clear by Lemma~\ref{dixoneasy}.
\end{proof}
\begin{corollary}
\label{ginfinitetest} Suppose that $G$ is nilpotent. Then $G$ is
finite if and only if $C_l$ in \eqref{chain4} is finite.
\end{corollary}

Corollary~\ref{ginfinitetest} gives another finiteness test for
completely reducible nilpotent subgroups $G$ of $\gln$; 
cf.~Subsection~\ref{decfinite}. This test requires that we are able to
decide finiteness of the finitely generated completely reducible
abelian matrix group $C_l$. To that end, the next result may be
useful.
\begin{lemma}
If $G$ is non-abelian nilpotent then $C_l$ has non-trivial
torsion.
\end{lemma}
\begin{proof}
Suppose that $C_l$ is torsion-free. Let $a\in Z_2(G)\setminus{\sf
Z}(G)$. Since $a^m\in {\sf Z}(G)$ for some $m$ by
Lemma~\ref{dixoneasy}, there exists $g\in G$ such that $[g,a]\in
{\sf Z}(G)$ has finite non-trivial order (dividing $m$). This
contradicts $[g,a]\in A_1\leq C_l$.
\end{proof}

Suppose that $G$ is finite. Then we can apply \cite[Lemma
2.23]{Large} to $G$. That is, we refine \eqref{chain4} to a
polycyclic series of $G$, then test nilpotency of $G$ via prime
factorization of the  cyclic quotients in the refined series, and
checking that factors for different primes commute. Hence the
algorithm ${\tt IsNilpotent}$ from \cite[Section 2]{Large} can be
employed for nilpotency testing of $G$. In the more general
setting we label this algorithm ${\tt IsFiniteNilpotent}$. This
algorithm, which accepts only finite $G\leq \gln$ as input, also
yields the Sylow decomposition of nilpotent $G$. Complexity
estimation 
of ${\tt
IsFiniteNilpotent}$ is undertaken in \cite[Section 2]{Large}.

Now we examine the case that $G$ is infinite. First we state a few
structural results.
\begin{lemma}
\label{periodres} Let $\pi$ be the set of primes less than or
equal to $n$. Suppose that $G$ is nilpotent. Then every element of
(the finite group) $G/{\sf Z}(G)$ has order divisible only by the
primes in $\pi$. Moreover, no element of $G/{\sf Z}(G)$ has order
divisible by $\mathrm{char}\, \mathbb{F}$.
\end{lemma}
\begin{proof}
It suffices to prove the lemma for irreducible $G$ 
(cf.~Remark~\ref{tooexplicit?}). 
Proofs of the irreducible case are
given in \cite[Chapter 7]{Suprunenko2}; see Corollary 1,
\mbox{p.206}, and Theorem 2, \mbox{p.} 216, of \cite{Suprunenko2}.
\end{proof}
\begin{corollary}
\label{sylpcentral} {\rm (Cf.~Example~\ref{finitefieldclassbd}.)} 
If $G$ is nilpotent then for
all primes $p>n$, a Sylow $p$-subgroup of $G$ is central.
\end{corollary}

Recall that a group $H$ is said to be {\it $p$-primary}, for a
prime $p$, if $H/{\sf Z}(H)$ is a $p$-group.
\begin{lemma}
If $G$ is nilpotent then $G$ is a product of $p$-primary groups
for $p\leq n$.
\end{lemma}

Now let $G$ be any subgroup of $\mathrm{GL}(n,\mathbb{F})$. Set
$m$ to be the $\mathbb{F}$-dimension of the enveloping algebra
$\langle G \rangle_{\mathbb{F}}$. Define the adjoint
representation $\mathrm{adj}:G
\rightarrow\mathrm{GL}(m,\mathbb{F})$ by $\mathrm{adj}(g): x
\mapsto$ $gxg^{-1}$, $x\in\langle G \rangle_{\mathbb{F}}$. Clearly
$\ker \mathrm{adj} = {\sf Z}(G)$. If $G$ is nilpotent and
completely reducible then $\mathrm{adj}(G)$ is a finite completely
reducible subgroup of $\mathrm{GL}(m,\mathbb{F})$, by
Lemma~\ref{periodres} and Maschke's theorem. If $G = \langle g_1,
\ldots , g_r\rangle$ then by \cite[Lemma 4.1]{Beals} we can
construct a basis of $\langle G \rangle_{\mathbb{F}}$ as a
straight-line program of length $m$ over $\{ g_1, \ldots , g_r\}$.
Then we 
calculate $\mathrm{adj}(G)$ by solving a system of linear
equations. The adjoint representation is another way of
transferring nilpotency testing to the case of finite  groups.

The results presented so far in this subsection lead to the
following algorithm to test nilpotency of $G$ using an abelian
series and the adjoint representation. The input generators of $G$
are diagonalizable, but $G$ cannot be assumed in advance to be
completely reducible. Also, as mentioned earlier, this algorithm
requires knowledge of an upper bound on nilpotency class of
nilpotent subgroups of $\mathrm{GL}(n,\mathbb{F})$.

\bigskip

\medskip

\hspace*{-1.5mm} ${\tt IsNilpotentAdjoint}(G)$

\vspace*{1.5mm}

Input: $G=\langle g_1, \ldots, g_r\rangle\leq \gln$, $g_i\in \gln$
diagonalizable.

Output:  a message `true' meaning that $G$ is nilpotent, or a
message `false' meaning that $G$ is

not nilpotent.

\vspace*{1.5mm}

for $i\in \{1, \ldots ,  r\}$

\hspace*{2mm} do $\bar{g}_i:= \mathrm{adj}(g_i)$;

$\bar{G}:= \langle \bar{g}_1, \ldots , \bar{g}_r\rangle$;

if $(\bar{g}_i)_u\neq 1$ for some $i$

\hspace*{2mm} then return `false';

else invoke ${\tt TestSeries}(\bar{G})$;

if $\bar{G}$ is infinite

\hspace*{2mm} then return `false';

else invoke ${\tt IsFiniteNilpotent}(\bar{G})$.

\bigskip

\medskip

For testing whether $\bar{G}$ is infinite, 
see Corollary~\ref{ginfinitetest}.

Parts of ${\tt IsNilpotentAdjoint}$ that use polynomial
factorization (e.g., the cutting procedure) have running
time dependent on the coefficient field.  Also, computation of
$\bar{G}=\mathrm{adj}(G)$ entails squaring the dimension in
worst-case; so may be time-consuming and efficient only for small
$n$.

If ${\tt IsNilpotentAdjoint}(G)$ returns `true' then the algorithm
furnishes additional information about $G$, such
as its decomposition into $p$-primary subgroups. Also, knowing a
generating set for $\bar{G}$ we can find a generating set for
${\sf Z}(G)=\ker \mathrm{adj}$ (by the Schreier method, or using a
presentation of $\bar{G}\cong  G/{\sf Z}(G)$ to pull back to
`normal subgroup generators', hence a generating set, of ${\sf
Z}(G)$; see before Lemma~\ref{easynormsgpgens}). If we can find a
polycyclic presentation of the finitely generated completely
reducible abelian matrix group ${\sf Z}(G)$, then this can be
combined with a polycyclic presentation of $G/{\sf Z}(G)$. Thus we
gain one more method to construct a polycyclic presentation of
$G$.

\subsection{Nilpotency testing via change of ground domain and abelian series}

Finally, we  outline the simplest and most effective combination
of our ideas for nilpotency testing of finitely generated matrix
groups, over a perfect field $\mathbb{F}$. 

\subsubsection{The algorithm}

${\tt IsNilpotentMatGroup}$ as below tests nilpotency over
an infinite field $\mathbb{F}$, via ${\tt Reduction}(G)$ (if
$\mathbb{F}$ is perfect), and applying a congruence homomorphism
$\psi_\varrho$ to $G_s$, where $\varrho$ satisfies the hypotheses
of Lemma~\ref{biguncor}. Nilpotency of $\psi_\varrho(G_s)$ is
tested using an abelian series \eqref{chain4} of
$\psi_\varrho(G_s)$ in $\glnq$: see
Subsection~\ref{firstuseofabseries}. 
If $G$ is nilpotent then
$(G_s)_\varrho \leq {\sf Z}(G_s)$, and this containment can be
tested via Lemma~\ref{easynormsgpgens}.

\bigskip

\medskip

\hspace*{-1.5mm} ${\tt IsNilpotentMatGroup}(G)$

\vspace*{1.5mm}

Input: $G=\langle g_1, \ldots, g_r\rangle\leq \gln$.

Output:  a message `true' meaning that $G$ is nilpotent, or a
message `false' meaning that $G$ is

not nilpotent.

\vspace*{1.5mm}

\begin{enumerate}
\item ${\tt Reduction}(G)$. \item Construct $\psi_\varrho(G_s)\leq
\glnq$. \item Test nilpotency of $\psi_\varrho(G_s)$. \item \label
{congcentral} Test whether $(G_s)_\varrho\leq {\sf Z}(G_s)$.
\end{enumerate}

\bigskip

\medskip

There are several advantages of the approach embodied in ${\tt
IsNilpotentMatGroup}$. First, by reducing the amount of
computation over the original field $\mathbb{F}$, we hope to
escape the unfortunate circumstances that may arise when
computing over infinite fields (e.g., a blow-up in the size of
matrix entries). Another issue relates to upper bounds on
nilpotency class. If $G$ is nilpotent then procedures used in
${\tt TestSeries}$ to construct the series \eqref{chain4} that
depend on a class bound for the potentially nilpotent group
$\psi_\varrho(G)$, such as ${\tt SecondCentralElement}$, 
are guaranteed to 
terminate more quickly than they would for an arbitrary nilpotent
subgroup of $\glnq$. For example, if $\mathbb{F}=\mathbb{Q}$ then
$\psi_\varrho(G)$ inherits from $G\leq \glq$ an upper bound $3n/2$
on nilpotency class; this can be compared with the general bound
\eqref{classbd} for $\glnq$ stated in
Example~\ref{finitefieldclassbd}.

It is desirable to retain complete reducibility in step (2) of
${\tt IsNilpotentMatGroup}$. That is, the Jordan decomposition
over the top field $\mathbb{F}$ is unavoidable; we do not 
want to repeat it in $\glnq$. Let $q$ be a power of the prime $p$.
For the input $\psi_\varrho((g_i)_s)$ to ${\tt TestSeries}$
to be diagonalizable, these elements must have order coprime
to $p$. Equivalently $\varrho$ should be chosen so that $f_i(X)$
and $f_i'(X)$ are coprime for all $i$, where $f_i(X)$ is the
minimal polynomial of $\psi_\varrho((g_i)_s)$ (note that $f_i(X)$
is the image $\psi_\varrho(h_i(X))$ of the minimal polynomial
$h_i(X)$ of $g_i$). Selection of $\varrho$ is a number theory
problem if $\mathbb{F}$ is a number field.
\begin{lemma}
\label{preimcent} If $G_s$ is nilpotent and $p>n$ then any
preimage of $(\psi_\varrho(G_s))_u$ in $G_s$ is central.
\end{lemma}
\begin{proof}
Let $g\in G_s$. Then $\psi_\varrho(g)_u = \psi_\varrho(g^l)$ for
some $l$, and  $\psi_\varrho(g^{lp^k})=1$ for some $k$; i.e.,
$g^{lp^k}\in\allowbreak (G_s)_\varrho\leq {\sf Z}(G_s)$. By
Corollary~\ref{sylpcentral}, $g^l \in {\sf Z}(G_s)$.
\end{proof}
Lemma~\ref{preimcent} indicates that we may reasonably expect
$\psi_\varrho(G_s)$ to be completely reducible if $\varrho$ is
chosen so that $p>n$. However, if $n$ is large then of course it
is advisable to work with $p\leq n$.

\subsubsection{Implementation and experimental results}

 Our implementation of ${\tt IsNilpotentMatGroup}$ for groups defined over
$\mathbb{Q}$ also includes an algorithm ${\tt
IsNilpotentMatGroupFF}$ for testing nilpotency over finite fields,
according to Subsection~\ref{firstuseofabseries}. To construct
congruence subgroups, ${\tt IsNilpotentMatGroup}$ uses some
functions from `Polenta'~\cite{Polenta}.

Table~\ref{onlytable} below samples performance of ${\tt
IsNilpotentMatGroup}$ for various input parameters: degree; size
of the field if finite, or size of generator entries if the field
is $\mathbb{Q}$; number of generators. The last column of
Table~\ref{onlytable} gives CPU time in the format minutes :
seconds : milliseconds. The computations were done on a Pentium 4
with 1.73 GHz under Windows, using {\sf GAP} 4. The standard {\sf
GAP} function ${\tt IsNilpotent}$ failed for all groups in
Table~\ref{onlytable}.

\begin{table}[h]
\begin{center}
\begin{tabular}{|c|c|c|c|c|}\hline
Group & Degree & Field & No. generators & Runtime  \\
\hline \hline $G_1$ & 9 & $5^6$  & 6  & $0:00:26.890$\\  \hline
$G_2$  & 127  & $2^7$  & 3 & $0:18:37.092$\\  \hline $G_3$  & 12 &
$5^6$ & 9  & $0:12:35.592$ \\  \hline $G_4$  & 30  & $11^4$  & 9 &
$0:17:39.749$\\  \hline $G_5$  & 63  & $2^6$  & 11  &
$0:18:25.842$ \\  \hline $G_6$  & 90  & $2^8$  & 54  &
$0:21:28.280$ \\  \hline $G_7$  & 96  & $5^4$  & 63  &
$0:35:27.546$  \\  \hline $G_8$  & 120 & $11^4$  & 27  &
$1:07:17.702$\\  \hline $G_9$  & 100  & $\mathbb{Q}$  & 12  &
$0:08:39.641$\\  \hline $G_{10}$  & 200  & $\mathbb{Q}$  & 27  &
$0:09:16.344$ \\  \hline $G_{11}$  & 128  & $\mathbb{Q}$  & 93  &
$0:14:07.398$ \\  \hline $G_{12}$  & 150 & $13^3$  & 2  &
$0:00:23.688$
\\  \hline
$G_{13}$  & 350 & $\mathbb{Q}$  & 4  & $0:00:55.047$
\\  \hline
$G_{14}$  & 25 & $\mathbb{Q}$  & 13  & $0:16:25.859$
\\  \hline
\end{tabular}
\end{center}

\vspace*{2.5mm}

\caption{Running times for nilpotency testing
algorithms}\label{onlytable}
\end{table}

As one might expect, the most challenging input groups are the
nilpotent groups, because for these we pass through all stages 
of the algorithms. 
On the other hand, if the input is not
nilpotent, then this is confirmed very quickly. For example, if
the input does not have an abelian series---if it
is not solvable---then the algorithm terminates at the ${\tt
TestSeries}$ stage (see Subsection~\ref{firstuseofabseries}).

Thus, for proper testing of our algorithms, we need an extensive
set of examples of nilpotent matrix groups. Constructing special
classes of nilpotent matrix groups is a problem of interest in its
own right. We have implemented an algorithm, ${\tt
MaximalAbsolutelyIrreducibleNilpotent}$- ${\tt MatGroup}(n,p,l)$,
that constructs absolutely irreducible maximal nilpotent subgroups
of $\mathrm{GL}(n,p^l)$.
If $r$ divides $p^l-1$ for each prime divisor $r$ of $n$
then such a subgroup of $\mathrm{GL}(n,p^l)$ is unique up to
conjugacy; otherwise, such subgroups do
not exist (see \cite[Chapter 7]{Suprunenko2}). If $n=r^a$ and 
$r$ divides $p^l-1$ then ${\tt
MaximalAbsolutelyIrreducibleNilpotentMatGroup}(n,p,l)$ returns the
group generated by a Sylow $r$-subgroup of $\mathrm{GL}(n,p^l)$,
and all scalars. For other $n$, this algorithm returns the
group generated by the scalars and a Kronecker product of Sylow
$r_i$-subgroups of $\mathrm{GL}(r_i^{a_i},p^l)$,
$n=\prod_{i=1}^kr_i^{a_i}$.

To check steps which rely on the Jordan decomposition, we
implemented another procedure, ${\tt ReducibleNilpotentMatGroup}$.
This procedure returns reducible but not completely reducible
nilpotent groups over finite fields and $\mathbb{Q}$.

The groups $G_i$ in Table~\ref{onlytable} for $i\leq 5$ are
absolutely irreducible nilpotent groups constructed by ${\tt
MaximalAbsolutelyIrreducibleNilpotentMatGroup}$.  The reducible
groups $G_6, G_7, G_8$, and $G_9, G_{10}, G_{11}$, are constructed
by ${\tt ReducibleNilpotentMatGroup}$. Finally, $G_{12}$,
$G_{13}$, and $G_{14}$ are non-nilpotent groups;
$G_{12}=\mathrm{GL}(150,13^3)$,
$G_{13}=\mathrm{GL}(350,\mathbb{Z})$, and $G_{14}$ is the group
${\tt POL}\! \underline{\hspace*{2.5mm}}{\tt PolExamples2(40)}$
from `Polenta', an infinite solvable subgroup of $\mathrm{GL}(25,
\mathbb{Q})$.

Other functions in `Nilmat' decide
finiteness, compute orders of finite nilpotent groups, find
the Sylow system of a nilpotent group over a finite field, and
test whether a nilpotent group is completely reducible. 
Additionally `Nilmat' contains a library of the nilpotent
primitive groups over finite fields (based on \cite{DFclass}).




\subsection*{Acknowledgment}
We are immensely indebted to Professor Bettina Eick for her
hospitality during our visits to Technische Universit\"{a}t
Braunschweig, for fruitful conversations then and afterwards, and
moreover for her very generous assistance with the {\sf GAP} 
implementation.



\newpage

\bibliographystyle{amsplain}


\end{document}